\documentclass[12pt]{amsart}
\usepackage{amsmath}
\usepackage{amsthm}
\usepackage{amssymb}
\usepackage{amscd}
\usepackage{graphics}

\usepackage{hyperref}
\usepackage{color}

\newcommand{\gothic}{\mathfrak}
\newcommand{\p}{{\gothic{p}}}

\newcommand{\Q}{{\Bbb{Q}}}
\newcommand{\m}{{\gothic{m}}}

\newcommand{\fkm}{\mathfrak m}

\newcommand{\vol}{\operatorname{vol}}

\newcommand{\depth}{\operatorname{depth}}

\newcommand{\pd}{\operatorname{pd}}

\newcommand{\syz}{\operatorname{syz}}
\newcommand{\Tor}{\operatorname{Tor}}

\newcommand{\IPD}{\operatorname{IPD}}

\newcommand{\Cl}{\operatorname{Cl}}

\newcommand{\length}{\ell}

\renewcommand{\bar}{\overline}
\renewcommand{\phi}{\varphi}

	\newcommand{\hk}{\text{e}_{\text{gHK}}}

\newcommand{\fhk}[1]{f_{gHK}^{#1}}

\DeclareMathOperator{\lc}{H}

\DeclareMathOperator{\Proj}{Proj}

\newcommand{\ses}[3]{0 \to {#1} \to {#2} \to {#3} \to 0}

\newcommand{\ps}[1]{{{}^{n}\!#1}}

\theoremstyle{plain}
\newtheorem{thm}{Theorem}
\newtheorem{cor}[thm]{Corollary}
\newtheorem{prop}[thm]{Proposition}
\newtheorem{lemma}[thm]{Lemma}
\newtheorem{notn}[thm]{Notation}

\newtheorem{eg}[thm]{Example}

\theoremstyle{definition}
\newtheorem{defn}[thm]{Definition}

\theoremstyle{remark}
\newtheorem{rmk}[thm]{Remark}
\newtheorem{question}[thm]{Question}

\newtheorem*{acknowledgement}{Acknowledgments}

\begin{document}
\title[Some computations of  generalized Hilbert-Kunz function and multiplicity]
{Some computations of  generalized Hilbert-Kunz function and multiplicity}
\author{Hailong Dao}
\address{Department of Mathematics\\
University of Kansas\\
 Lawrence, KS 66045-7523 USA}
\email{hdao@ku.edu}

\author{Kei-ichi Watanabe}
\address{Department of Mathematics, College of Human and Science\\
 Nihon University\\
  Setagaya, Tokyo, 156-0045, Japan}
\email{watanabe@math.chs.nihon-u.ac.jp}

\date{\today}
\thanks{The first author is partially supported by NSF grant 1104017. The second author was partially supported by  JSPS Grant-in-Aid for Scientific Research (C) Grant Number 26400053.}

\bibliographystyle{amsplain}

\numberwithin{thm}{section}
\numberwithin{equation}{section}
\begin{abstract}
Let $R$ be a  local ring of characteristic $p>0$ which is $F$-finite and has perfect residue field. We compute the  generalized Hilbert-Kunz invariant (studied in \cite{DS, EY}) for certain modules over several classes of rings: hypersurfaces of finite representation type,  toric rings, $F$-regular rings.
\end{abstract}

\maketitle
\tableofcontents

%%%%%%%%%%%%%%%%%%%%%%%%%%%%%%%%%%%%%%%
\section{Introduction}\label{intro}
%%%%%%%%%%%%%%%%%%%%%%%%%%%%%%%%%%%%%%

Let $R$ be a  local ring of characteristic $p>0$ which is $F$-finite and has perfect residue field.  Let $M$ a finitely generated $R$-module. Let $F_R^n(M) = M\otimes_R \ps R$
 denote the $n$-fold iteration of the Frobenius functor given by base change along %KW
 the Frobenius endomorphism. 
Let $\dim R = d$ and $q=p^n$. This paper constitutes a further study of the following:  
\[\fhk M(n) :=\length(\lc^0_\m(F^n(M)))\]
and
\[\hk(M) := \lim_{n \to \infty} \frac{\fhk M(n)}{p^{nd}},\]
which are called  the {\it generalized Hilbert-Kunz function} and {\it generalized Hilbert-Kunz multiplicity} of $M$, respectively.  These notions were first defined by Epstein-Yao in \cite{EY} and were studied in details in \cite{DS}. For instance, it is now known that $\hk(M)$ exists for all modules over a Cohen-Macaulay isolated singularity. 

It is a non-trivial and interesting problem to compute even the classical Hilbert-Kunz multiplicity. In this note we focus on computing $\fhk M$ and the limit $\hk(M)$ for certain modules in a number of cases: when $R$ is a normal domain of dimension $2$ (section \ref{dim2}), a hypersurface of finite representation type (section \ref{ft}) and when $R$ is a toric ring (section \ref{toric}). We also point out a connection between the generalized Hilbert-Kunz limits and tight closure theory in section \ref{fr}. Namely,  over $F$-regular rings, these limits detect depths of the module $M$ and all of its pull-back along iterations of Frobenius.

\begin{acknowledgement}
The authors would like to thank Jack Jeffries and Jonathan Monta\~no for some helpful discussions on the subject of $j$-multiplicity. We also thank MSRI and Nihon University for providing excellent environment for our collaboration. 
\end{acknowledgement}

\section{Dimension two}\label{dim2}

In this section, we prove certain preliminary facts about behavior of $\fhk M$ when $R$ is normal and  $M=R/I$ where $I$ is reflexive. We then apply them to give a formula for $\hk(R/I)$ when $I$ represents %KW
 a torsion element in the class group of $R$. 

\begin{lemma}\label{keylem}
Let $R$ be a local normal  domain  of dimension at least $2$ and $I$ a reflexive  ideal that is locally free on the punctured spectrum. Then $$\length(\lc^0_{\m}(R/I^{[q]})) = \length(I^{(q)}/I^{[q]}) =  \length(I^{(q)}/I^{q}) +  \length(I^{q}/I^{[q]}) =  \length(\lc^0_{\m}(R/I^{q}))+ \length(I^{q}/I^{[q]}) $$   
\end{lemma}

\begin{proof}
Apply local cohomology functor to the sequence: $$\ses{\frac{I^{(q)}}{I^{[q]}}}{\frac{R}{I^{[q]}}}{\frac{R}{I^{(q)}}}.$$ Note that $\lc^0_{\m}(R/I^{(q)})=0$ and $\length(I^{(q)}/I^{q}), \length(I^{q}/I^{[q]}) <\infty$ as the ideals coincide on the punctured spectrum. 
\end{proof}

\begin{prop}
Let $R$ be a local  normal domain of dimension $2$ and $I$ be a reflexive ideal. Then $\hk(R/I)= 0$   if and only if $I$ is principal.
\end{prop}

\begin{proof}
Only one direction needs to be checked. Suppose $\hk(R/I)= 0$. Let $\mu()$ denote the minimal number of generators of an $R$-module. We have that $ \length(\lc^0_\m(R/I^{[q]}))\geq \length(\lc^0_{\m}(R/I^{q}))$ by Lemma \ref{keylem}. It follows that $\limsup \frac{\length(\lc^0_{\m}(R/I^{q}))}{q^2}=0$, so $I$ has analytic spread one by \cite[Theorem 4.7]{JK}, thus $[I]$ is principal. 
\end{proof}

\begin{rmk}
The number   $\limsup \frac{\length(\lc^0_{\m}(R/I^{n}))}{n^2}=0$ is known as the epsilon multiplicity of $I$, $\epsilon(I)$. It has now been proved to exist as a limit under mild conditions, it see \cite{Cut}. Lemma \ref{keylem} says that $\hk(R/I) \geq \epsilon(I)$.
\end{rmk}

\begin{lemma}\label{lemcal}
Let $R$ be a local  normal domain of dimension $2$ and $I$ a reflexive  ideal. Assume that $[I]$ is torsion in $\Cl(R)$. Then $\length(\lc^0_\m(R/I^n))$ has quasi-polynomial behavior for $n$ large enough.
\end{lemma}

\begin{proof}
Let $r$ be  some integer such that of $r[I]=0$ in $\Cl(R)$. Then $I^r = I_1\cap I_2$, where $I_1$ is the determinant of $I^r$ and thus principal, and $I_2$ is $\m$-primary. Let $I_1=(x)$ we then have $I^r = xJ$ where $J=I_2:x$. Note that $J$ is $\m$-primary. For any integer $n$, let $n=ar+b$. We have that $I^n =I^{ar+b} = x^aJ^aI^b$. Then $$\lc^0_\m(R/I^n) \cong \lc^1_\m(I^n)   \cong \lc^1_\m(J^aI^b) \cong \lc^0_\m(R/J^aI^b)$$ 

To calculate the last term we use: $$\ses{I^b/J^aI^b}{R/J^aI^b}{R/I^b} $$
The leftmost term has finite length, thus what we want is equal to $\length(I^b/J^aI^b)+ \length( \lc^0_\m(R/I^b))$. Since $b$ is periodic and $a$ grows linearly with $n$, what we claimed follows. Note that the limit if $\length(\lc^0_\m(R/I^n))/n^2$ is equal to $e(J)/2r^2$.

\end{proof}

\section{The finite representation type case}\label{ft}

We now describe how to compute $\hk(M)$ when $M$ is a module of positive depth over a Gorenstein local ring of finite Cohen-Macaulay type. We first need some definitions. 

\begin{defn}\label{ftdef}
Let $R$ be a Gorenstein complete local ring of finite Cohen-Macaulay type with perfect residue field (in particular, $R$ must be a hypersurface singularity, see \cite{Yo}). Let $X_1,\cdots, X_n$ be all the indecomposable non-free Cohen-Macaulay modules. 

We define the {\it stable Cohen-Macaulay type} of $M$ to be the vector $(u_1,\cdots, u_n)$ with $X=\oplus X_i^{u_i}$, here $X$ is a Cohen-Macaulay approximation $\ses MNX$ where $\pd_RN<\infty$. This is well-defined since $R$ is complete. As $R$ is also a hypersurface, by taking syzygy one can see that $X$ is stably equivalenct %KW
to the  e-syzygy of $M$ where $e=2\dim R$. 

We also define  $v_j = \lim_{n\to \infty} \frac{\#(\ps R,  X_j)}{q^n}$, where $\#(\ps R,  X_j)$ is the number of copies of $X_j$ in the decomposition of $\ps R$. This limit exists by \cite{Sei,Yao}.
\end{defn}

\begin{prop}
Using the set up of Definition \ref{ftdef}. Let $M$ be an $R$-module of positive depth. One has:
$$\hk(M) = \sum_{1\leq i,j\leq n} u_iv_j\length(\Tor_1^R(X_i,X_j)) = \sum_{1\leq i,j\leq n} u_iv_j\length(\Tor_2^R(X_i,X_j))$$
\end{prop}

\begin{proof}
Take a MCM approximation $\ses MNX$ and tensor with $\ps R$, we get $$0 \to \Tor^R_1(X, \ps R) \to M\otimes \ps R \to N\otimes \ps R$$
Note that $\depth(N\otimes \ps R) = \depth N = \depth M>0$ and $\Tor^R_1(X, \ps R)$ has finite length as $R$ must have isolated singularity, we get that $\length(\lc_\m^0(M\otimes \ps R) = \length(\Tor^R_1(X, \ps R))$. The first equality is now obvious. 

For the second equality we just need  that $\length(\Tor^R_1(X, \ps R)) = \length(\Tor^R_2(X, \ps R))$ by \cite{Dao}.

\end{proof}

\begin{eg}
Let $R=k[[x,y,z]]/(xy-z^r)$ where $k$ is a perfect field of characteristic $p>0$. $R$ has finite type with $X_i= (x,z^i)$, $1\leq i\leq r-1$. It is not hard to check that $\length(\Tor_1^R(X_i,X_j)) = \min\{i,j,r-i,r-j\}$. Also, it is known that $v_j=1/r$. So for a module $M$ with positive depth and stable CM type $(u_1,\dots, u_n)$ one gets:
$$\hk(M) = \frac{1}{r}\sum_{1\leq i,j\leq r-1} u_i\min\{i,j,r-i,r-j\}$$
\end{eg}

\section{The toric  case}\label{toric}

In this section, we show how to compute the generalized  Hilbert-Kunz multiplicity of
$R/I$, where $R$ is a normal toric ring and $I$ is an of $R$ generated  by monomials. We fix the following notation.

\begin{notn}

Let $k$ be a field of chaaracteristic $p$ and $M \cong {\Bbb Z}^d$ be a lattice and 
$M_{\Bbb R} = M\otimes_{\Bbb Z}{\Bbb R}$. Let 
$\sigma \subset  M_{\Bbb R}$ be a strongly convex rational polyhedral cone 
and $R= k[ \sigma\cap M] = k[ X^m | m\in \sigma\cap M]$ be a normal toric ring.  
Let $I =( X^{m_1}, \ldots , X^{m_s})$ be a {monomial} ideal of $R$.
We put $\Gamma_I$ the convex hull of $\bigcup_{i=1}^s [m_i + \sigma]$ and 
$W_I =  \bigcup_{i=1}^s [m_i + \sigma]$.  {
We define a subset $LC_{I}$ of $M_{\Bbb R}$ by  
$$m \in LC_{I} \  \text{iff} \  [m+\sigma] \setminus [m+\sigma]\cap W_I  \ \text{has finite volume} $$ }
\end{notn}

{
\begin{prop}
With the notation above:
\begin{enumerate}
\item For $m\in M$, $x^m \in I: J^{\infty}$ (where $J$ is the maximal ideal) if and only if $m \in LC_I$. 
\item $LC_{I^{[q]}} = qLC_I$ and $W_{I^{[q]}} = qW_I$. 
\item $LC_I \setminus W_I$ is a bounded region in $M_{\Bbb R}$.  

\end{enumerate}
\end{prop}

\begin{proof}
It is clear  from  the definition that $m\in LC_I$ iff $x^mJ^t \subseteq I$ for $t\gg0$. (2) is also clear. For (3), note that the region in question is defined by finitely many half planes. Thus, if it has infinite  volume, there will be $q$ big enough such that  $LC_I \setminus W_I$ contains infinitely many points in $\frac{1}{q}{\Bbb Z}^d$. In other words, there are infinitely many integral points in $LC_{I^{[q]}} \setminus W_{I^{[q]}}$. But the integral points in that region simply correspond to  the monomials in $\lc^0_{\m}(R/I^{[q]})$, a contradiction.

\end{proof}

\begin{rmk}
In the picture below, $LC_I \setminus W_I$ can be seen as the combination of the red and green regions. 
\end{rmk}

\begin{thm}
Let $R$ and $I$ be as above. Then $\hk (R/I) = \vol (LC_I \setminus W_I)$.  In particular, $\hk(R/I)\in \Q$. 
\end{thm}
\center
\includegraphics{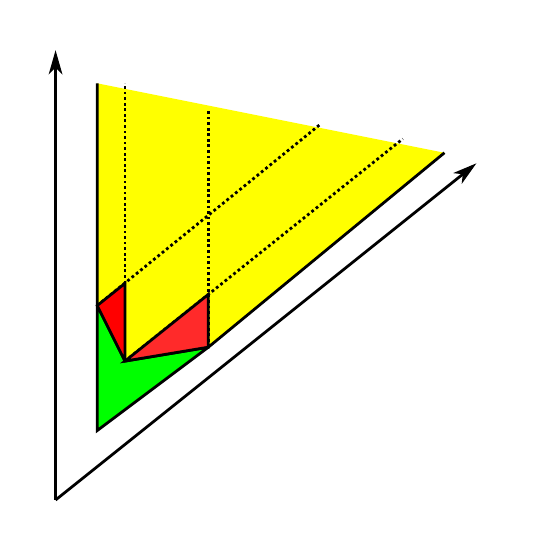}

\begin{proof}

The previous Proposition tells us that $m\in LC_{I^{[q]}}$ iff $m/q \in LC_I$, from which the result follows.
\end{proof}
}

We demonstrate the ideas of the last Theorem with two concrete examples.

\begin{prop}
Let $R=k[[x^r, x^{r-1}y, \cdots,y^r]]$ be isomorphic to  the $r$-Veronese of $k[[x,y]]$ and $I_m =(x^r, x^{r-1}y, \cdots,x^{r-m}y^m) \subset R$ be the one of the reflexive ideals of $R$ (note that $I$ corresponds to the element $\bar m \in \mathbb Z/(r) \cong \Cl(R)$). Then
$$\hk(R/I_m) = \frac{m(m+1)}{2r}$$

\end{prop}

\begin{proof}
Let $I=I_m$. Note that $I^r = (x^{r^2}, \cdots, x^{(r-m)r}y^{rm}) = x^{(r-m)r}\m^m$.
We use Lemmas \ref{keylem} and \ref{lemcal} to calculate the relevant lengths.
It follows that  $\lim \length(\lc^0_\m(R/I^q))/q^d = e(\m^m)/2r^2 = m^2/2r.$

The second part involves   $\length(I^{q}/I^{[q]})$. The monomials that are in $I^q$ but not in $I^{[q]}$ are contained in the right triangles whose hypotenuses are the intervals $({iq},(r-i)q), ((i-1)q, (r-i+1)q)$ with $i=r,\cdots, r-m+1$. It is clear that the number of such monomials, which is the length we want, is of order $mq^2/2r$. So the second term contribute $m/2r$ to the limit. We conclude that: $$\hk(R/I) = m^2/2r+ m/2r =m(m+1)/2r $$

\end{proof}

\begin{prop}
Let $R=k[[x,y,z]]/(xy-z^r)$  and $I_m =(x,z^m) \subset R$ ($m<r$). Then
$$\hk(R/I_m) = \frac{m(r-m)}{r}$$

\end{prop}

\begin{proof}
Let $I=I_m$. Note that $I^r = x^m (x^{r-m}, x^{r-m-1}z^m ,\ldots , z^{r-m} y^{m-1}, y^m)=x^mJ$.
We again use Lemmas \ref{keylem} and \ref{lemcal}. 

If we assign point  $x \to (r, -1), y \to (0,1), z %KW
\to (1,0)$, the points corresponding to $(x^{r-m}, x^{r-m-1}z^r ,\ldots , z^{r-m} y^{m-1}, y^m)$ are
$( r(r-m), -(r-m)), \ldots,(r-m, m-1), (0, m)$, lying on a line of
slope $1/(r-m)$.  %KW 
This line and the cone defined by $x\ge 0$ and $y \ge - x/r$  
form a triangle of area $r(r-m)/2$, this means the multiplicity of the ideal
 $J$ is $rm(r-m)$.
Hence, $\lim \length(\lc^0_\m(R/I^q))/q^d=m(r-m)/2r$.

On the other hand, $\length(I^q / I^{[q]})$ corresponds to the triangle whose vertices are
 $(qm, 0), (qr, -q)$ and $(qr , -q(r-m)/r)$. The area is $m(r-m)q^2/r$. Summing up we have $\hk(R/I) = \frac{m(r-m)}{r}$.
\end{proof}

\section{The F-regular case}\label{fr}

Lastly, we study a connection between generalized Hilbert-Kunz multiplicity and tight closure theory.  We first recall the following criterion for tight closure  due to  Hochster-Huneke. 

\begin{lemma}\label{HH}
Let $R$ be equidimensional and either complete or essentially of finite type over a field  and $N \subseteq  L\subseteq G$ be finitely generated $R$-modules such that $L/N$ has finite length. Then $\hk(G/N) \geq \hk(G/L)$, and equality occurs if and only if $L\subseteq N_G^*$. 
\end{lemma}

We now want to show:

\begin{prop}\label{Freg}
Let $R$ be $F$-regular (i.e, all ideals are tightly closed) and $M$  be a finitely generated $R$-module. The following are equivalent:

\begin{enumerate}
\item $\hk(M) = 0$.
\item $\depth F^n(M) >0$ for all $n\geq 0$.
\end{enumerate}
\end{prop}

\begin{proof}
We only need to show (1) implies (2). It is harmless to complete $R$ and $M$ (see Exercise 4.1 in \cite{Hu}). Suppose there exists $n\geq 0$ such that $\depth F^n(M)=0$, we need to prove that $\hk(M)>0$.  Replacing $M$ by $F^n(M)$ if neccessary, we may assume $\depth M=0$. 
Now take a short exact sequence $\ses{N}{G}{M}$ where $G$ is free. 
Let $x \in G$ represent an element in the socle of $M$, we know that $L = (N,x) \nsubseteq N = N_G^*$, thus  $\hk(M)>\hk(G/L)\geq 0$ by Lemma \ref{HH}. 
\end{proof}

\begin{rmk}
When $R$ is strongly $F$-regular, one can prove the above Proposition as follows. The assumption means that we have  decompositions of $R$-modules $\ps R = R^{a_q}\oplus M_q$ and $c= \lim_{n\to \infty} \frac{a_q}{q^d}>0$.  Then it is clear that $\hk(M)\geq c\length(\lc^0_{\m}(M))$, so the non-trivial direction (1) implies (2) is now easy to see. 

\end{rmk}

\begin{cor}
Let $R$ be $F$-regular of dimension at least $2$ and $I$ be a reflexive ideal that is locally free on the punctured spectrum. Then $\hk (R/I)=0$ if and only if $I$ is principal.
\end{cor}

\begin{proof}
By Proposition \ref{Freg} we only need to show that $\depth R/I^{[q]}=0$ for some $q$. But suppose it is not the case, then Lemma \ref{keylem} implies that $I^q=I^{[q]}$ for all $q$, thus the analytic spread is one. 

\end{proof}

Before moving on we recall the following limits studied  in \cite{DS}. Let $i\geq 0$ be an integer. Let 
$$\hk^i (M) := \lim_{n \to \infty} \frac{\ell(\lc^i_{\m}(F^n(M)))}{p^{nd}}$$%KW

Let $\IPD(M)$ denote the set of prime ideals $\p$ such that $\pd_{R_\p}M_\p=\infty$. 

\begin{lemma}\label{lclem}
Let $R$ be of depth $d$. Let $N$ be an $R$-module such that $\IPD(N) \subseteq\{\m\}$. Let $M$ be a $t$-syzygy of $N$. Then $\lc_\m^{i+t}(F^n(M)) \cong \lc^i_\m(F^n(N))$ for $0\leq i\leq d-t-1$.

\end{lemma}

\begin{proof}
We begin with tensoring the exact sequence $\ses{\syz N}{F}{N}$ with  $\ps R$ to get
\[0 \to \Tor_1^R (N, \ps R) \to F^n(\syz N) \xrightarrow{f} F^n(F) \to F^n(N) \to 0\]
which we break into:
$$ \ses{\Tor_1^R (N, \ps R)}{F^n(\syz N)}{C}$$
and
$$\ses{C}{F^n(F) }{F^n(N)} $$
Note that $\Tor_1^R (N, \ps R)$ has finite length, so the long sequence of local cohomology for the first sequence gives $\lc_\m^{i}(F^n(\syz N)) \cong \lc_\m^{i}(C)$ for $i>0$. For the second sequence, we have that $\lc_\m^{i}(F^n(N)) \cong  \lc_\m^{i+1}(C)$ for $0 \leq i\leq d-2$. Thus $$\lc_\m^{i}(F^n(N)) \cong  \lc_\m^{i+1}(F^n(\syz N))$$ for $0 \leq i\leq d-2$. A simple induction finishes the proof. 

\end{proof}

\begin{thm}\label{eithm}
Let $R$ be $F$-regular of dimension $d\geq 2$ and $0\leq a \leq b \leq d-1$ be integers. Let $M$ be an $R$-module that is locally free on the punctured spectrum and $\depth M \geq a$. The following are equivalent: 
\begin{enumerate}
\item $\hk^i (M)=0$ for  $a\leq i\leq b$. 
\item $\lc^i_\m (F^n(M))=0$ for all $a\leq i \leq b$ and all $n\geq 0$.
\end{enumerate}
\end{thm}

\begin{proof}
We use induction on $b-a$. It is enough to prove the case $b=a$, since the conclusion implies that $\depth M\geq a+1$, and we can replace $a$ by $a+1$. As $\depth M\geq a$, we can pushforward $a$ times and write $M$ as $\syz^a N$ for some module $N$.   Proposition \ref{Freg} and Lemma \ref{lclem} finish the proof. 
\end{proof}

\begin{cor}
Let $R$ be $F$-regular and $I$ be a reflexive ideal that is locally free on the punctured spectrum. If $[I]$ is torsion in the class group of $R$ then $I$ is Cohen-Macaulay. 
\end{cor}

\begin{proof}
We can assume  $R$ has dimension is at least $3$. We just note that the double dual of $F^n(I)$, $F^n(I)^{**}$, is isomorphic to $I^{(q)}$, which corresponds to the element $q[I]$ in $\Cl(R)$. The natural map   $F^n(I) \to F^n(I)^{**}$ has kernel and cokernel of finite length. It follows that $\lc^i_\m(F^n(I)) \cong \lc^i_\m(F^n(I)^{**}) \cong \lc^i_\m(I^{(q)})$ for $i\geq 2$. But the isomorphism classes of $I^{(q)}$ will be periodic as $[I]$ is torsion. Thus $\hk^i(I)=0$ for $2\leq i \leq d-1$, and by Theorem \ref{eithm}, $\lc^i_\m(I)=0$ for $2\leq i\leq d-1$, which is all we need to prove. 
\end{proof}

\begin{rmk}
If the order of $[I]$ is prime to the characteristic of $R$, the result was first proved, without condition that $I$ is locally free on the punctured spectrum, for strongly $F$-regular rings in \cite{Wa}. The condition on the order of $[I]$ was removed in \cite[Corollary 3.3]{PS}. All of these results will be extended in \cite{DSe}, with a more direct approach.  
\end{rmk}

\end{document}